\documentclass[11pt]{amsart}%
\usepackage{amsmath}
\usepackage{graphicx}
\usepackage{amsfonts}
\usepackage{amssymb}%
\newtheorem{theorem}{Theorem}[section]
\theoremstyle{plain}

\newtheorem{lemma}{Lemma}[section]

\newtheorem{proposition}{Proposition}[section]

\numberwithin{equation}{section}
\begin{document}
\title[Hessian Estimates]{Hessian estimates for the sigma-2 equation in dimension three}
\author{Micah Warren}
\author{Yu YUAN}
\address{Department of Mathematics, Box 354350\\
University of Washington\\
Seattle, WA 98195}
\email{mwarren@math.washington.edu, yuan@math.washington.edu}
\thanks{Y.Y. is partially supported by an NSF grant.}
\date{\today}

\begin{abstract}
We derive a priori interior Hessian estimates for the special Lagrangian
equation $\sigma_{2}=1$ in dimension three.

\end{abstract}
\maketitle

\section{\bigskip Introduction}

In this article, we derive an \textit{interior a priori} Hessian estimate for
the $\sigma_{2}$ equation
\begin{equation}
\sigma_{2}\left(  D^{2}u\right)  =\lambda_{1}\lambda_{2}+\lambda_{2}%
\lambda_{3}+\lambda_{3}\lambda_{1}=1 \label{sigma2}%
\end{equation}
in dimension three, where $\lambda_{i}$ are the eigenvalues of the Hessian
$D^{2}u.$ We attack (\ref{sigma2}) via its special Lagrangian equation form%
\begin{equation}
\sum_{i=1}^{n}\arctan\lambda_{i}=\Theta\label{EsLag}%
\end{equation}
with $n=3$ and $\Theta=\pi/2.$ Equation (\ref{EsLag}) stems from the special
Lagrangian geometry [HL]. The Lagrangian graph $\left(  x,Du\left(  x\right)
\right)  \subset\mathbb{R}^{n}\times\mathbb{R}^{n}$ is called special when the
phase or the argument of the complex number $\left(  1+\sqrt{-1}\lambda
_{1}\right)  \cdots\left(  1+\sqrt{-1}\lambda_{n}\right)  $ is constant
$\Theta,$ and it is special if and only if $\left(  x,Du\left(  x\right)
\right)  $ is a (volume minimizing) minimal surface in $\mathbb{R}^{n}%
\times\mathbb{R}^{n}$ [HL, Theorem 2.3, Proposition 2.17].

We state our result in the following

\begin{theorem}
Let $u$ be a smooth solution to (\ref{sigma2}) on $B_{R}(0)\subset
\mathbb{R}^{3}.$ Then we have%
\[
|D^{2}u(0)|\leq C(3)\exp\left[  C(3)\max_{B_{R}(0)}|Du|^{3}/R^{3}\right]  .
\]

\end{theorem}

By Trudinger's [T] gradient estimates for $\sigma_{k}$ equations, we can bound
$D^{2}u$ in terms of the solution $u$ in $B_{2R}\left(  0\right)  $ as%
\[
|D^{2}u(0)|\leq C(3)\exp\left[  C(3)\max_{B_{2R}(0)}|u|^{3}/R^{6}\right]  .
\]
One immediate consequence of the above estimates is a Liouville type result
for global solutions with quadratic growth to (\ref{sigma2}), namely any such
a solution must be quadratic (cf. [Y1], [Y2]). Another consequence is the
regularity (analyticity) of the $C^{0}$ viscosity solutions to (\ref{sigma2})
or (\ref{EsLag}) with $n=3$ and $\Theta=\pm\pi/2.$

In the 1950's, Heinz [H] derived a Hessian bound for the two dimensional
Monge-Amp\`{e}re equation, $\sigma_{2}(D^{2}u)=\lambda_{1}\lambda_{2}%
=\det(D^{2}u)=1,$ which is equivalent to (\ref{EsLag}) with $n=2\ $and
$\Theta=\pm\pi/2.$ In the 1970's Pogorelov [P] constructed his famous
counterexamples, namely irregular solutions to three dimensional
Monge-Amp\`{e}re equations $\sigma_{3}(D^{2}u)=\lambda_{1}\lambda_{2}%
\lambda_{3}=\det(D^{2}u)=1;$ see generalizations of the counterexamples for
$\sigma_{k}$ equations with $k\geq3$ in [U1]. Hessian estimates for solutions
with certain strict convexity constraints to Monge-Amp\`{e}re equations and
$\sigma_{k}$ equation ($k\geq2$) were derived by Pogorelov [P2] and Chou-Wang
[CW] respectively using the Pogorelov technique. Urbas [U2][U3], also Bao and
Chen obtained (pointwise) Hessian estimates in term of certain integrals of
the Hessian, for $\sigma_{k}$ equations and special Lagrangian equation (1.1)
with $n=3,\ \Theta=\pi$ respectively.

The heuristic idea of the proof of Theorem 1.1 is as follows. The function
$b=\ln\sqrt{1+\lambda_{\max}^{2}}$ is subharmonic so that $b$ at any point is
bounded by its integral over a ball around the point on the minimal surface by
Michael-Simon's mean value inequality [MS]. This special choice of $b$ is not
only subharmonic, but even stronger, satisfies a Jacobi inequality. This
Jacobi inequality leads to a bound on the integral of $b$ by the volume of the
ball on the minimal surface. Taking advantage of the divergence form of the
volume element of the minimal Lagrangian graph, we bound the volume in terms
of the height of the special Lagrangian graph, which is the gradient of the
solution to equation (\ref{EsLag}).

Now the challenging regularity problem for sigma-2 equations in dimension four
and higher still remains open to us.

\textbf{Notation. }$\partial_{i}=\frac{\partial}{\partial_{x_{i}}}%
,\ \partial_{ij}=\frac{\partial^{2}}{\partial x_{i}\partial x_{j}}%
,\ u_{i}=\partial_{i}u,\ u_{ji}=\partial_{ij}u$ etc., but $\lambda_{1}%
,\cdots,\lambda_{n}$ and $b_{1}=\ln\sqrt{1+\lambda_{1}^{2}}$, $b_{2}=\left(
\ln\sqrt{1+\lambda_{1}^{2}}+\ln\sqrt{1+\lambda_{2}^{2}}\right)  /2\ $do not
represent the partial derivatives. Further, $h_{ijk}$ will denote (the second
fundamental form)
\[
h_{ijk}=\frac{1}{\sqrt{1+\lambda_{i}^{2}}}\frac{1}{\sqrt{1+\lambda_{j}^{2}}%
}\frac{1}{\sqrt{1+\lambda_{k}^{2}}}u_{ijk}.
\]
when $D^{2}u$ is diagonalized. Finally $C\left(  n\right)  $ will denote
various constants depending only on dimension $n.$

\section{Preliminary inequalities}

Taking the gradient of both sides of the special Lagrangian equation
(\ref{EsLag}), we have
\begin{equation}
\sum_{i,j}^{n}g^{ij}\partial_{ij}\left(  x,Du\left(  x\right)  \right)
=0,\label{Emin}%
\end{equation}
where $\left(  g^{ij}\right)  $ is the inverse of the induced metric
$g=\left(  g_{ij}\right)  =I+D^{2}uD^{2}u$ on the surface $\left(  x,Du\left(
x\right)  \right)  \subset\mathbb{R}^{n}\times\mathbb{R}^{n}.$ Simple
geometric manipulation of (\ref{Emin}) yields the usual form of the minimal
surface equation
\[
\bigtriangleup_{g}\left(  x,Du\left(  x\right)  \right)  =0,
\]
where the Laplace-Beltrami operator of the metric $g$ is given by
\[
\bigtriangleup_{g}=\frac{1}{\sqrt{\det g}}\sum_{i,j}^{n}\partial_{i}\left(
\sqrt{\det g}g^{ij}\partial_{j}\right)  .
\]
Because we are using harmonic coordinates $\bigtriangleup_{g}x=0,$ we see that
$\bigtriangleup_{g}$ also equals the linearized operator of the special
Lagrangian equation (\ref{EsLag}) at $u,$%
\[
\bigtriangleup_{g}=\sum_{i,j}^{n}g^{ij}\partial_{ij}.
\]
The gradient and inner product with respect to the metric $g$ are
\begin{align*}
\nabla_{g}v &  =\left(  \sum_{k=1}^{n}g^{1k}v_{k},\cdots,\sum_{k=1}^{n}%
g^{nk}v_{k}\right)  ,\\
\left\langle \nabla_{g}v,\nabla_{g}w\right\rangle _{g} &  =\sum_{i,j=1}%
^{n}g^{ij}v_{i}w_{j},\ \ \text{in particular \ }\left\vert \nabla
_{g}v\right\vert ^{2}=\left\langle \nabla_{g}v,\nabla_{g}v\right\rangle _{g}.
\end{align*}
We begin with some geometric calculations.

\begin{lemma}
Let $u$ be a smooth solution to (\ref{EsLag}). Suppose that the Hessian
$D^{2}u$ is diagonalized and the eigenvalue $\lambda_{1}$ is distinct from all
other eigenvalues of $D^{2}u$ at point $p.$ Set $b_{1}=\ln\sqrt{1+\lambda
_{1}^{2}}$ near $p.$ \ Then we have at $p$%
\begin{equation}
\left\vert \nabla_{g}b_{1}\right\vert ^{2}=\sum_{k=1}^{n}\lambda_{1}%
^{2}h_{11k}^{2}\label{gradientb1}%
\end{equation}
and%
\begin{gather}
\bigtriangleup_{g}b_{1}=\nonumber\\
(1+\lambda_{1}^{2})h_{111}^{2}+\sum_{k>1}\left(  \frac{2\lambda_{1}}%
{\lambda_{1}-\lambda_{k}}+\frac{2\lambda_{1}^{2}\lambda_{k})}{\lambda
_{1}-\lambda_{k}}\right)  h_{kk1}^{2}\label{lapb1a}\\
+\sum_{k>1}\left[  1+\frac{2\lambda_{1}}{\lambda_{1}-\lambda_{k}}%
+\frac{\lambda_{1}^{2}\left(  \lambda_{1}+\lambda_{k}\right)  }{\lambda
_{1}-\lambda_{k}}\right]  h_{11k}^{2}\label{lapb1b}\\
+\sum_{k>j>1}2\lambda_{1}\left[  \frac{1+\lambda_{k}^{2}}{\lambda_{1}%
-\lambda_{k}}+\frac{1+\lambda_{j}^{2}}{\lambda_{1}-\lambda_{j}}+(\lambda
_{j}+\lambda_{k})\right]  h_{kj1}^{2}.\label{lapb1c}%
\end{gather}

\end{lemma}

\begin{proof}
[Proof]We first compute the derivatives of the smooth function $b_{1}$ near
$p.$ We may implicitly differentiate the characteristic equation
\[
\det(D^{2}u-\lambda_{1}I)=0
\]
near any point where $\lambda_{1}$ is distinct from the other eigenvalues.
Then we get at $p$
\begin{align*}
\partial_{e}\lambda_{1}  &  =\partial_{e}u_{11},\\
\partial_{ee}\lambda_{1}  &  =\partial_{ee}u_{11}+\sum_{k>1}2\frac{\left(
\partial_{e}u_{1k}\right)  ^{2}}{\lambda_{1}-\lambda_{k}},
\end{align*}
with arbitrary unit vector $e\in\mathbb{R}^{n}.$

Thus we have (\ref{gradientb1}) at $p$%
\[
|\nabla_{g}b_{1}|^{2}=\sum_{k=1}^{n}g^{kk}\left(  \frac{\lambda_{1}}%
{1+\lambda_{1}^{2}}\partial_{k}u_{11}\right)  ^{2}=\sum_{k=1}^{n}\lambda
_{1}^{2}h_{11k}^{2},
\]
where we used the notation $h_{ijk}=\sqrt{g^{ii}}\sqrt{g^{jj}}\sqrt{g^{kk}%
}u_{ijk}.$

From
\[
\partial_{ee}b_{1}=\partial_{ee}\ln\sqrt{1+\lambda_{1}^{2}}=\frac{\lambda_{1}%
}{1+\lambda_{1}^{2}}\partial_{ee}\lambda_{1}+\frac{1-\lambda_{1}^{2}}{\left(
1+\lambda_{1}^{2}\right)  ^{2}}\left(  \partial_{e}\lambda_{1}\right)  ^{2},
\]
we conclude that at $p$
\[
\partial_{ee}b_{1}=\frac{\lambda_{1}}{1+\lambda_{1}^{2}}\left[  \partial
_{ee}u_{11}+\sum_{k>1}2\frac{\left(  \partial_{e}u_{1k}\right)  ^{2}}%
{\lambda_{1}-\lambda_{k}}\right]  +\frac{1-\lambda_{1}^{2}}{\left(
1+\lambda_{1}^{2}\right)  ^{2}}\left(  \partial_{e}u_{11}\right)  ^{2},
\]
and
\begin{align}
\bigtriangleup_{g}b_{1}  &  =\sum_{\gamma=1}^{n}g^{\gamma\gamma}%
\partial_{\gamma\gamma}b_{1}\nonumber\\
&  =\sum_{\gamma=1}^{n}g^{\gamma\gamma}\frac{\lambda_{1}}{1+\lambda_{1}^{2}%
}\left(  \partial_{\gamma\gamma}u_{11}+\sum_{k>1}2\frac{\left(  u_{1k\gamma
}\right)  ^{2}}{\lambda_{1}-\lambda_{k}}\right)  +\sum_{\gamma=1}^{n}%
\frac{1-\lambda_{1}^{2}}{\left(  1+\lambda_{1}^{2}\right)  ^{2}}%
g^{\gamma\gamma}u_{11\gamma}^{2}. \label{E4thorder}%
\end{align}

Next we substitute the fourth order derivative terms $\partial_{\gamma\gamma
}u_{11}$ in the above by lower order derivative terms.\ Differentiating the
minimal surface equation (\ref{Emin}) $\sum_{\alpha,\beta=1}^{n}g^{\alpha
\beta}u_{j\alpha\beta}=0,$\ we obtain
\begin{align}
\bigtriangleup_{g}u_{ij}  &  =\sum_{\alpha,\beta=1}^{n}g^{\alpha\beta
}u_{ji\alpha\beta}=\sum_{\alpha,\beta=1}^{n}-\partial_{i}g^{_{_{\alpha\beta}}%
}u_{j\alpha\beta}=\sum_{\alpha,\beta,\gamma,\delta=1}^{n}g^{\alpha\gamma
}\partial_{i}g_{\gamma\delta}g^{\delta\beta}u_{j\alpha\beta}\nonumber\\
&  =\sum_{\alpha,\beta=1}^{n}g^{\alpha\alpha}g^{\beta\beta}(\lambda_{\alpha
}+\lambda_{\beta})u_{\alpha\beta i}u_{\alpha\beta j}, \label{lapujj}%
\end{align}
where we used
\[
\partial_{i}g_{_{\gamma\delta}}=\partial_{i}(\delta_{\gamma\delta}%
+\sum_{\varepsilon=1}^{n}u_{\gamma\varepsilon}u_{\varepsilon\delta}%
)=u_{\gamma\delta i}(\lambda_{\gamma}+\lambda_{\delta})
\]
with diagonalized $D^{2}u.$ Plugging (\ref{lapujj}) with $i=j=1$ in
(\ref{E4thorder}), we have at $p$%

\begin{align*}
\bigtriangleup_{g}b_{1}  &  =\frac{\lambda_{1}}{1+\lambda_{1}^{2}}\left[
\sum_{\alpha,\beta=1}^{n}g^{\alpha\alpha}g^{\beta\beta}(\lambda_{\alpha
}+\lambda_{\beta})u_{\alpha\beta1}^{2}+\sum_{\gamma=1}^{n}\sum_{k>1}%
2\frac{u_{1k\gamma}^{2}}{\lambda_{1}-\lambda_{k}}g^{\gamma\gamma}\right] \\
&  +\sum_{\gamma=1}^{n}\frac{1-\lambda_{1}^{2}}{\left(  1+\lambda_{1}%
^{2}\right)  ^{2}}g^{\gamma\gamma}u_{11\gamma}^{2}\\
&  =\lambda_{1}\sum_{\alpha,\beta=1}^{n}(\lambda_{\alpha}+\lambda_{\beta
})h_{\alpha\beta1}^{2}+\sum_{k>1}\sum_{\gamma=1}^{n}\frac{2\lambda_{1}\left(
1+\lambda_{k}^{2}\right)  }{\lambda_{1}-\lambda_{k}}h_{\gamma k1}^{2}%
+\sum_{\gamma=1}^{n}\left(  1-\lambda_{1}^{2}\right)  h_{11\gamma}^{2},
\end{align*}
where we used the notation $h_{ijk}=\sqrt{g^{ii}}\sqrt{g^{jj}}\sqrt{g^{kk}%
}u_{ijk}.$ Regrouping those terms $h_{\heartsuit\heartsuit1},\ h_{11\heartsuit
},$ and $h_{\heartsuit\clubsuit1}$ in the last expression, we have
\begin{gather*}
\bigtriangleup_{g}b_{1}=\left(  1-\lambda_{1}^{2}\right)  h_{111}^{2}%
+\sum_{\alpha=1}^{n}2\lambda_{1}\lambda_{\alpha}h_{\alpha\alpha1}^{2}%
+\sum_{k>1}\frac{2\lambda_{1}\left(  1+\lambda_{k}^{2}\right)  }{\lambda
_{1}-\lambda_{k}}h_{kk1}^{2}\\
+\sum_{k>1}2\lambda_{1}(\lambda_{k}+\lambda_{1})h_{k11}^{2}+\sum_{k>1}\left(
1-\lambda_{1}^{2}\right)  h_{11k}^{2}+\sum_{k>1}\frac{2\lambda_{1}\left(
1+\lambda_{k}^{2}\right)  }{\lambda_{1}-\lambda_{k}}h_{1k1}^{2}\\
+\sum_{k>j>1}2\lambda_{1}(\lambda_{j}+\lambda_{k})h_{jk1}^{2}+\sum
_{\substack{j,k>1,\\j\neq k}}\frac{2\lambda_{1}\left(  1+\lambda_{k}%
^{2}\right)  }{\lambda_{1}-\lambda_{k}}h_{jk1}^{2}.
\end{gather*}
After simplifying the above expression, we have the second formula in Lemma 2.1.
\end{proof}

\begin{lemma}
Let $u$ be a smooth solution to (\ref{EsLag}) with $n=3$ and $\Theta\geq
\pi/2.$ Suppose that the ordered eigenvalues $\lambda_{1}\geq\lambda_{2}%
\geq\lambda_{3}$ of the Hessian $D^{2}u$ satisfy $\lambda_{1}>\lambda_{2}$ at
point $p.$ Set%
\[
b_{1}=\ln\sqrt{1+\lambda_{\max}^{2}}=\ln\sqrt{1+\lambda_{1}^{2}}.
\]
Then we have at $p$%
\begin{equation}
\bigtriangleup_{g}b_{1}\geq\frac{1}{3}|\nabla_{g}b_{1}|^{2}.\label{Jacobi-3d}%
\end{equation}

\end{lemma}

\begin{proof}
[Proof ]We assume that the Hessian $D^{2}u$ is diagonalized at point $p.$

Step 1. Recall $\theta_{i}=\arctan\lambda_{i}\in(-\pi/2,\pi/2)$ and
\ $\theta_{1}+\theta_{2}+\theta_{3}=\Theta\geq\pi/2.$ It is easy to see that
$\theta_{1}\geq\theta_{2}>0$ and $\theta_{i}+\theta_{j}\geq0$ for any pair.
Consequently $\lambda_{1}\geq\lambda_{2}>0\ $and$\ \lambda_{i}+\lambda_{j}%
\geq0$ for any pair of distinct eigenvalues. It follows that (\ref{lapb1c}) in
the formula for $\bigtriangleup_{g}b_{1}$ is positive, then from
(\ref{lapb1a}) and (\ref{lapb1b}) we have the inequality\
\begin{equation}
\bigtriangleup_{g}b_{1}\geq\lambda_{1}^{2}\left(  h_{111}^{2}+\sum_{k>1}%
\frac{2\lambda_{k}}{\lambda_{1}-\lambda_{k}}h_{kk1}^{2}\right)  +\lambda
_{1}^{2}\sum_{k>1}\left(  1+\frac{2\lambda_{k}}{\lambda_{1}-\lambda_{k}%
}\right)  h_{11k}^{2}. \label{cor21z}%
\end{equation}

Combining (\ref{cor21z}) and (\ref{gradientb1}) gives \
\begin{gather}
\bigtriangleup_{g}b_{1}-\frac{1}{3}|\nabla_{g}b_{1}|^{2}\geq\nonumber\\
\lambda_{1}^{2}\left(  \frac{2}{3}h_{111}^{2}+\sum_{k>1}\frac{2\lambda_{k}%
}{\lambda_{1}-\lambda_{k}}h_{kk1}^{2}\right)  +\lambda_{1}^{2}\sum_{k>1}%
\frac{2\left(  \lambda_{1}+2\lambda_{k}\right)  }{3\left(  \lambda_{1}%
-\lambda_{k}\right)  }h_{11k}^{2}. \label{cor21a}%
\end{gather}

Step 2. \ We show that the last term in (\ref{cor21a}) is nonnegative. Note
that $\lambda_{1}+2\lambda_{k}\geq\lambda_{1}+2\lambda_{3}.$ We only need to
show that $\lambda_{1}+2\lambda_{3}\geq0$ in the case that $\lambda_{3}<0$ or
equivalently $\theta_{3}<0.$ From $\theta_{1}+\theta_{2}+\theta_{3}=\Theta
\geq$ $\pi/2,$ we have
\[
\frac{\pi}{2}>\theta_{3}+\frac{\pi}{2}=\left(  \frac{\pi}{2}-\theta
_{1}\right)  +\left(  \frac{\pi}{2}-\theta_{2}\right)  +\Theta-\frac{\pi}%
{2}\geq2\left(  \frac{\pi}{2}-\theta_{1}\right)  .
\]
It follows that%
\[
-\frac{1}{\lambda_{3}}=\tan\left(  \theta_{3}+\frac{\pi}{2}\right)
>2\tan\left(  \frac{\pi}{2}-\theta_{1}\right)  =\frac{2}{\lambda_{1}},
\]
then%
\begin{equation}
\lambda_{1}+2\lambda_{3}>0. \label{l_1+2l_3>0}%
\end{equation}
\ 

Step 3. \ We show that the first term in (\ref{cor21a}) is nonnegative by
proving
\begin{equation}
\frac{2}{3}h_{111}^{2}+\frac{2\lambda_{2}}{\lambda_{1}-\lambda_{2}}h_{221}%
^{2}+\frac{2\lambda_{3}}{\lambda_{1}-\lambda_{3}}h_{331}^{2}\geq
0.\label{cor21c}%
\end{equation}
We only need to show it for $\lambda_{3}<0.$ Directly from the minimal surface
equation (\ref{Emin})%
\[
h_{111}+h_{221}+h_{331}=0,
\]
we bound
\[
h_{331}^{2}=\left(  h_{111}+h_{221}\right)  ^{2}\leq\left(  \frac{2}{3}%
h_{111}^{2}+\frac{2\lambda_{2}}{\lambda_{1}-\lambda_{2}}h_{221}^{2}\right)
\left(  \frac{3}{2}+\frac{\lambda_{1}-\lambda_{2}}{2\lambda_{2}}\right)  .
\]
It follows that
\begin{gather*}
\frac{2}{3}h_{111}^{2}+\frac{2\lambda_{2}}{\lambda_{1}-\lambda_{2}}h_{221}%
^{2}+\frac{2\lambda_{3}}{\lambda_{1}-\lambda_{3}}h_{331}^{2}\geq\\
\left(  \frac{2}{3}h_{111}^{2}+\frac{2\lambda_{2}}{\lambda_{1}-\lambda_{2}%
}h_{221}^{2}\right)  \left[  1+\frac{2\lambda_{3}}{\lambda_{1}-\lambda_{3}%
}\left(  \frac{3}{2}+\frac{\lambda_{1}-\lambda_{2}}{2\lambda_{2}}\right)
\right]  .
\end{gather*}
The last term becomes%
\[
1+\frac{2\lambda_{3}}{\lambda_{1}-\lambda_{3}}\left(  \frac{3}{2}%
+\frac{\lambda_{1}-\lambda_{2}}{2\lambda_{2}}\right)  =\frac{\sigma_{2}%
}{\left(  \lambda_{1}-\lambda_{3}\right)  \lambda_{2}}>0.
\]
The above inequality is from the observation
\[
\operatorname{Re}%
%TCIMACRO{\dprod \limits_{i=1}^{3}}%
%BeginExpansion
{\displaystyle\prod\limits_{i=1}^{3}}
%EndExpansion
\left(  1+\sqrt{-1}\lambda_{i}\right)  =1-\sigma_{2}\leq0
\]
for $3\pi/2>\theta_{1}+\theta_{2}+\theta_{3}=\Theta\geq$ $\pi/2.$ Therefore
(\ref{cor21c}) holds.

We have proved the pointwise Jacobi inequality (\ref{Jacobi-3d}) in Lemma 2.2.
\end{proof}

\begin{lemma}
Let $u$ be a smooth solution to (\ref{EsLag}) with $n=3$ and $\Theta\geq
\pi/2.$ Suppose that the ordered eigenvalues $\lambda_{1}\geq\lambda_{2}%
\geq\lambda_{3}$ of the Hessian $D^{2}u$ satisfy $\lambda_{2}>\lambda_{3}$ at
point $p.$ Set%
\[
b_{2}=\frac{1}{2}\left(  \ln\sqrt{1+\lambda_{1}^{2}}+\ln\sqrt{1+\lambda
_{2}^{2}}\right)  .
\]
\ Then $b_{2}$ satisfies at $p$
\begin{equation}
\bigtriangleup_{g}b_{2}\geq0.\label{subharmonicb2}%
\end{equation}
Further, suppose that $\lambda_{1}\equiv$ $\lambda_{2}$ in a neighborhood of
$p.$\ Then $b_{2}$ satisfies at $p$
\begin{equation}
\bigtriangleup_{g}b_{2}\geq\frac{1}{3}\left\vert \nabla_{g}b_{2}\right\vert
^{2}.\label{jacobi-b2}%
\end{equation}

\end{lemma}

\begin{proof}
\ We assume that Hessian $D^{2}u$ is diagonalized at point $p.$ We may use
Lemma 2.1 to obtain expressions for both \ $\bigtriangleup_{g}\ln
\sqrt{1+\lambda_{1}^{2}}$ and $\bigtriangleup_{g}\ln\sqrt{1+\lambda_{2}^{2}},$
whenever the eigenvalues of $D^{2}u$ are distinct. From (\ref{lapb1a}),
(\ref{lapb1b}), and (\ref{lapb1c}),\ we have
\begin{gather}
\bigtriangleup_{g}\ln\sqrt{1+\lambda_{1}^{2}}+\bigtriangleup_{g}\ln
\sqrt{1+\lambda_{2}^{2}}=\label{lapb2}\\
(1+\lambda_{1}^{2})h_{111}^{2}+\sum_{k>1}\frac{2\lambda_{1}(1+\lambda
_{1}\lambda_{k})}{\lambda_{1}-\lambda_{k}}h_{kk1}^{2}+\sum_{k>1}\left[
1+\lambda_{1}^{2}+2\lambda_{1}\left(  \frac{1+\lambda_{1}\lambda_{k}}%
{\lambda_{1}-\lambda_{k}}\right)  \right]  h_{11k}^{2}\nonumber\\
+2\lambda_{1}\left[  \frac{1+\lambda_{3}^{2}}{\lambda_{1}-\lambda_{3}}%
+\frac{1+\lambda_{2}^{2}}{\lambda_{1}-\lambda_{2}}+(\lambda_{3}+\lambda
_{2})\right]  h_{321}^{2}\nonumber\\
+(1+\lambda_{2}^{2})h_{222}^{2}+\sum_{k\neq2}\frac{2\lambda_{2}(1+\lambda
_{2}\lambda_{k})}{\lambda_{2}-\lambda_{k}}h_{kk2}^{2}+\sum_{k\neq2}\left[
1+\lambda_{2}^{2}+2\lambda_{2}\left(  \frac{1+\lambda_{2}\lambda_{k}}%
{\lambda_{2}-\lambda_{k}}\right)  \right]  h_{22k}^{2}\nonumber\\
+2\lambda_{2}\left[  \frac{1+\lambda_{3}^{2}}{\lambda_{2}-\lambda_{3}}%
+\frac{1+\lambda_{1}^{2}}{\lambda_{2}-\lambda_{1}}+(\lambda_{3}+\lambda
_{1})\right]  h_{321}^{2}.\text{ \ }\nonumber
\end{gather}
The function $b_{2}$ is symmetric in $\lambda_{1}$ and $\lambda_{2},$ thus
$b_{2}$ is smooth even when $\lambda_{1}=\lambda_{2},$ provided that
$\lambda_{2}>\lambda_{3}.$ \ We simplify (\ref{lapb2}) to the following, which
holds by continuity wherever $\lambda_{1}\geq\lambda_{2}>\lambda_{3}.$%
\[
2\bigtriangleup_{g}b_{2}=
\]%
\begin{align}
&  (1+\lambda_{1}^{2})h_{111}^{2}+(3+\lambda_{2}^{2}+2\lambda_{1}\lambda
_{2})h_{221}^{2}+\left(  \frac{2\lambda_{1}}{\lambda_{1}-\lambda_{3}}%
+\frac{2\lambda_{1}^{2}\lambda_{3}}{\lambda_{1}-\lambda_{3}}\right)
h_{331}^{2}\label{lapb2a}\\
&  +(3+\lambda_{1}^{2}+2\lambda_{1}\lambda_{2})h_{112}^{2}+(1+\lambda_{2}%
^{2})h_{222}^{2}+\left(  \frac{2\lambda_{2}}{\lambda_{2}-\lambda_{3}}%
+\frac{2\lambda_{2}^{2}\lambda_{3}}{\lambda_{2}-\lambda_{3}}\right)
h_{332}^{2}\label{lapb2b}\\
&  +\left[  \frac{3\lambda_{1}-\lambda_{3}+\lambda_{1}^{2}(\lambda_{1}%
+\lambda_{3})}{\lambda_{1}-\lambda_{3}}\right]  h_{113}^{2}+\left[
\frac{3\lambda_{2}-\lambda_{3}+\lambda_{2}^{2}(\lambda_{2}+\lambda_{3}%
)}{\lambda_{2}-\lambda_{3}}\right]  h_{223}^{2}\label{lapb2c}\\
&  +2\left[  1+\lambda_{1}\lambda_{2}+\lambda_{2}\lambda_{3}+\lambda
_{3}\lambda_{1}+\frac{\lambda_{1}\left(  1+\lambda_{3}^{2}\right)  }%
{\lambda_{1}-\lambda_{3}}+\frac{\lambda_{2}\left(  1+\lambda_{3}^{2}\right)
}{\lambda_{2}-\lambda_{3}}\right]  h_{123}^{2}. \label{lapb2d}%
\end{align}
Using the relations $\lambda_{1}\geq\lambda_{2}>0,\ \lambda_{i}+\lambda
_{j}>0,$ and $\sigma_{2}\geq1$ derived in the proof of Lemma 2.2, we see that
(\ref{lapb2d}) and (\ref{lapb2c}) are nonnegative. We only need to justify the
nonnegativity of (\ref{lapb2a}) and (\ref{lapb2b}) for $\lambda_{3}<0.$ From
the minimal surface equation (\ref{Emin}), we know%
\[
h_{332}^{2}=\left(  h_{112}+h_{222}\right)  ^{2}\leq\left[  (\lambda_{1}%
^{2}+2\lambda_{1}\lambda_{2})h_{112}^{2}+\lambda_{2}^{2}h_{222}^{2}\right]
\left(  \frac{1}{\lambda_{1}^{2}+2\lambda_{1}\lambda_{2}}+\frac{1}{\lambda
_{2}^{2}}\right)  .
\]
It follows that%
\begin{align*}
(\ref{lapb2b})  &  \geq(\lambda_{1}^{2}+2\lambda_{1}\lambda_{2})h_{112}%
^{2}+\lambda_{2}^{2}h_{222}^{2}+\frac{2\lambda_{2}^{2}\lambda_{3}}{\lambda
_{2}-\lambda_{3}}h_{332}^{2}\\
&  \geq\left[  (\lambda_{1}^{2}+2\lambda_{1}\lambda_{2})h_{112}^{2}%
+\lambda_{2}^{2}h_{222}^{2}\right]  \left[  1+\frac{2\lambda_{2}^{2}%
\lambda_{3}}{\lambda_{2}-\lambda_{3}}\left(  \frac{1}{\lambda_{1}^{2}%
+2\lambda_{1}\lambda_{2}}+\frac{1}{\lambda_{2}^{2}}\right)  \right]  .
\end{align*}
The last term becomes%
\begin{align*}
&  \frac{2\lambda_{2}^{2}\lambda_{3}}{\lambda_{2}-\lambda_{3}}\left(
\frac{\lambda_{2}-\lambda_{3}}{2\lambda_{2}^{2}\lambda_{3}}+\frac{1}%
{\lambda_{1}^{2}+2\lambda_{1}\lambda_{2}}+\frac{1}{\lambda_{2}^{2}}\right) \\
&  =\frac{\lambda_{2}}{\lambda_{2}-\lambda_{3}}\left[  \frac{\sigma_{2}%
}{\lambda_{1}\lambda_{2}}-\frac{\lambda_{3}}{\left(  \lambda_{1}+2\lambda
_{2}\right)  }\right]  \geq0.
\end{align*}
Thus (\ref{lapb2b}) is nonnegative. Similarly (\ref{lapb2a}) is nonnegative.
We have proved (\ref{subharmonicb2}).

Next we prove (\ref{jacobi-b2}), still assuming $D^{2}u$ is diagonalized at
point $p.$ Plugging in $\lambda_{1}=\lambda_{2}$ into (\ref{lapb2a}),
(\ref{lapb2b}), and (\ref{lapb2c}), we get
\[
2\bigtriangleup_{g}b_{2}\geq
\]%
\begin{align*}
&  \lambda_{1}^{2}\left(  h_{111}^{2}+3h_{221}^{2}+\frac{2\lambda_{3}}%
{\lambda_{1}-\lambda_{3}}h_{331}^{2}\right)  \\
&  +\;\lambda_{1}^{2}\left(  3h_{112}^{2}+h_{222}^{2}+\frac{2\lambda_{3}%
}{\lambda_{1}-\lambda_{3}}h_{332}^{2}\right)  \\
&  +\lambda_{1}^{2}\left(  \frac{\lambda_{1}+\lambda_{3}}{\lambda_{1}%
-\lambda_{3}}\right)  \left(  h_{113}^{2}+h_{223}^{2}\right)  .
\end{align*}
Differentiating the eigenvector equations in the neighborhood where
$\lambda_{1}\equiv$ $\lambda_{2}$
\[
\left(  D^{2}u\right)  U=\frac{\lambda_{1}+\lambda_{2}}{2}U,\ \left(
D^{2}u\right)  V=\frac{\lambda_{1}+\lambda_{2}}{2}V,\ \text{and}\ \left(
D^{2}u\right)  W=\lambda_{3}W,
\]
we see that $u_{11e}=u_{22e}$ for any $e\in\mathbb{R}^{3}$ at point $p.$ Using
the minimal surface equation (\ref{Emin}), we then have
\[
h_{11k}=h_{22k}=-\frac{1}{2}h_{33k}%
\]
at point $p.$ Thus%
\[
\bigtriangleup_{g}b_{2}\geq\lambda_{1}^{2}\left[  2\left(  \frac{\lambda
_{1}+\lambda_{3}}{\lambda_{1}-\lambda_{3}}\right)  h_{111}^{2}+2\left(
\frac{\lambda_{1}+\lambda_{3}}{\lambda_{1}-\lambda_{3}}\right)  h_{112}%
^{2}+\left(  \frac{\lambda_{1}+\lambda_{3}}{\lambda_{1}-\lambda_{3}}\right)
h_{113}^{2}\right]  .
\]

The gradient$|\nabla_{g}b_{2}|^{2}$ has the expression at $p$
\[
|\nabla_{g}b_{2}|^{2}=\sum_{k=1}^{3}g^{kk}\left(  \frac{1}{2}\frac{\lambda
_{1}}{1+\lambda_{1}^{2}}\partial_{k}u_{11}+\frac{1}{2}\frac{\lambda_{2}%
}{1+\lambda_{2}^{2}}\partial_{k}u_{22}\right)  ^{2}=\sum_{k=1}^{3}\lambda
_{1}^{2}h_{11k}^{2}.
\]
Thus at $p$%
\begin{gather*}
\bigtriangleup_{g}b_{2}-\frac{1}{3}|\nabla_{g}b_{2}|^{2}\geq\\
\lambda_{1}^{2}\left\{  \left[  2\left(  \frac{\lambda_{1}+\lambda_{3}%
}{\lambda_{1}-\lambda_{3}}\right)  -\frac{1}{3}\right]  h_{111}^{2}+\left[
2\left(  \frac{\lambda_{1}+\lambda_{3}}{\lambda_{1}-\lambda_{3}}\right)
-\frac{1}{3}\right]  h_{112}^{2}+\left(  \frac{\lambda_{1}+\lambda_{3}%
}{\lambda_{1}-\lambda_{3}}-\frac{1}{3}\right)  h_{113}^{2}\right\} \\
\geq0,
\end{gather*}
where we again used $\lambda_{1}+2\lambda_{3}>0$ from (\ref{l_1+2l_3>0}). We
have proved (\ref{jacobi-b2}) of Lemma 2.3.
\end{proof}

\begin{proposition}
\label{PIJacobi}Let $u$ be a smooth solution to the special\ Lagrangian
equation (\ref{EsLag}) with $n=3$ and $\Theta=\pi/2$ on $B_{4}\left(
0\right)  \subset\mathbb{R}^{3}.$ Set
\[
b=\max\left\{  \ln\sqrt{1+\lambda_{\max}^{2}},\ K\right\}
\]
with $K=1+\ln\sqrt{1+\tan^{2}\left(  \frac{\pi}{6}\right)  .}$ Then $b$
satisfies the integral Jacobi inequality
\begin{equation}
\int_{B_{4}}-\left\langle \nabla_{g}\varphi,\nabla_{g}b\right\rangle
_{g}dv_{g}\geq\frac{1}{3}\int_{B_{4}}\varphi\left\vert \nabla_{g}b\right\vert
^{2}dv_{g}\label{IJacobi-3}%
\end{equation}
for all non-negative $\varphi\in C_{0}^{\infty}\left(  B_{4}\right)  .$
\end{proposition}

\begin{proof}
If $b_{1}=\ln\sqrt{1+\lambda_{\max}^{2}}$ is smooth everywhere, then the
pointwise Jacobi inequality (\ref{Jacobi-3d}) in Lemma 2.2 \ already implies
the integral Jacobi (\ref{IJacobi-3}).\ It is known that $\lambda_{\max}$ is
always a Lipschitz function of the entries of the Hessian $D^{2}u.$ \ Now $u$
is smooth in $x,$ so $b_{1}=\ln\sqrt{1+\lambda_{\max}^{2}}$ is Lipschitz in
terms of $x.$ If \ $b_{1}$ (or equivalently $\lambda_{\max})\;$is not smooth,
then the first two largest eigenvalues $\lambda_{1}\left(  x\right)  $ and
$\lambda_{2}\left(  x\right)  $ coincide, and $b_{1}\left(  x\right)
=b_{2}\left(  x\right)  ,$ where $b_{2}\left(  x\right)  $ is the average
$b_{2}=\left(  \ln\sqrt{1+\lambda_{1}^{2}}+\ln\sqrt{1+\lambda_{2}^{2}}\right)
/2.$ \ We prove the integral Jacobi inequality (\ref{IJacobi-3}) for a
possibly singular $b_{1}\left(  x\right)  $ in two cases. Set
\[
S=\left\{  x|\ \lambda_{1}\left(  x\right)  =\lambda_{2}\left(  x\right)
\right\}  .
\]

Case 1. $S$ has measure zero. For small $\tau>0,$ let%
\begin{align*}
\Omega &  =B_{4}\backslash\left\{  x|\ b_{1}\left(  x\right)  \leq K\right\}
=B_{4}\backslash\left\{  x|\ b\left(  x\right)  =K\right\}  \\
\Omega_{1}\left(  \tau\right)   &  =\left\{  x|\ b\left(  x\right)
=b_{1}\left(  x\right)  >b_{2}\left(  x\right)  +\tau\right\}  \cap\Omega\\
\Omega_{2}\left(  \tau\right)   &  =\left\{  x|\ b_{2}\left(  x\right)  \leq
b\left(  x\right)  =b_{1}\left(  x\right)  <b_{2}\left(  x\right)
+\tau\right\}  \cap\Omega.
\end{align*}
Now $b\left(  x\right)  =b_{1}\left(  x\right)  $ is smooth in $\overline
{\Omega_{1}\left(  \tau\right)  }.$ We claim that $b_{2}\left(  x\right)  $ is
smooth in $\overline{\Omega_{2}\left(  \tau\right)  }.$ We know $b_{2}\left(
x\right)  $ is smooth wherever $\lambda_{2}\left(  x\right)  >\lambda
_{3}\left(  x\right)  .$ If (the \ Lipschitz) $b_{2}\left(  x\right)  $ is not
smooth at $x_{\ast}\in\overline{\Omega_{2}\left(  \tau\right)  },$ then%
\begin{align*}
\ln\sqrt{1+\lambda_{3}^{2}} &  =\ln\sqrt{1+\lambda_{2}^{2}}\geq\ln
\sqrt{1+\lambda_{1}^{2}}-2\tau\\
&  \geq\ln\sqrt{1+\tan^{2}\left(  \frac{\pi}{6}\right)  }+1-2\tau,
\end{align*}
by the choice of $K.$ For small enough $\tau$, we have $\lambda_{2}%
=\lambda_{3}>\tan\left(  \frac{\pi}{6}\right)  $ and a contradiction
\[
\left(  \theta_{1}+\theta_{2}+\theta_{3}\right)  \left(  x_{\ast}\right)
>\frac{\pi}{2}.
\]

Note that
\begin{gather*}
\int_{B_{4}}-\left\langle \nabla_{g}\varphi,\nabla_{g}b\right\rangle
_{g}dv_{g}=\int_{\Omega}-\left\langle \nabla_{g}\varphi,\nabla_{g}%
b\right\rangle _{g}dv_{g}\\
=\lim_{\tau\rightarrow0^{+}}\left[  \int_{\Omega_{1}\left(  \tau\right)
}-\left\langle \nabla_{g}\varphi,\nabla_{g}b\right\rangle _{g}dv_{g}%
+\int_{\Omega_{2}\left(  \tau\right)  }-\left\langle \nabla_{g}\varphi
,\nabla_{g}\left(  b_{2}+\tau\right)  \right\rangle _{g}dv_{g}\right]  .
\end{gather*}
By the smoothness of $b$ in $\Omega_{1}\left(  \tau\right)  $ and $b_{2}$ in
$\Omega_{2}\left(  \tau\right)  ,$ and also inequalities (\ref{Jacobi-3d}) and
(\ref{subharmonicb2}), we have
\begin{align*}
&  \int_{\Omega_{1}\left(  \tau\right)  }-\left\langle \nabla_{g}%
\varphi,\nabla_{g}b\right\rangle _{g}dv_{g}+\int_{\Omega_{2}\left(
\tau\right)  }-\left\langle \nabla_{g}\varphi,\nabla_{g}\left(  b_{2}%
+\tau\right)  \right\rangle _{g}dv_{g}\\
&  =\int_{\partial\Omega_{1}\left(  \tau\right)  }-\varphi\text{ }%
\partial_{\gamma_{g}^{1}}b\text{ }dA_{g}+\int_{\Omega_{1}\left(  \tau\right)
}\varphi\bigtriangleup_{g}b_{1}dv_{g}\\
&  +\int_{\partial\Omega_{2}\left(  \tau\right)  }-\varphi\partial_{\gamma
_{g}^{2}}\left(  b_{2}+\tau\right)  dA_{g}+\int_{\Omega_{2}\left(
\tau\right)  }\varphi\bigtriangleup_{g}\left(  b_{2}+\tau\right)  dv_{g}\\
&  \geq\int_{\partial\Omega_{1}\left(  \tau\right)  }-\varphi\text{ }%
\partial_{\gamma_{g}^{1}}b\text{ }dA_{g}+\int_{\partial\Omega_{2}\left(
\tau\right)  }-\varphi\partial_{\gamma_{g}^{2}}\left(  b_{2}+\tau\right)
dA_{g}+\frac{1}{3}\int_{\Omega_{1}\left(  \tau\right)  }\varphi\left\vert
\nabla_{g}b_{1}\right\vert ^{2}dv_{g},
\end{align*}
where $\gamma_{g}^{1}$ and $\gamma_{g}^{2}$ are the outward co-normals of
$\partial\Omega_{1}\left(  \tau\right)  $ and $\partial\Omega_{2}\left(
\tau\right)  $ with respect to the metric $g.$

Observe that if $b_{1}$ is not smooth on any part of $\partial\Omega
\backslash\partial B_{4}$, which is the $K$-level set of $b_{1,}$ then on this
portion $\partial\Omega\backslash\partial B_{4}$ is also the $K$-level set of
$b_{2},$ which is smooth near this portion. Applying Sard's theorem, we can
perturb $K$ so that $\partial\Omega$ is piecewise $C^{1}.$ Applying Sard's
theorem again, we find a subsequence of positive $\tau$ going to $0,$ so that
the boundaries $\partial\Omega_{1}\left(  \tau\right)  $ and $\partial
\Omega_{2}\left(  \tau\right)  $ are piecewise $C^{1}.$

Then, we show the above boundary integrals are non-negative. The boundary
integral portion along $\partial\Omega$ is easily seen non-negative, because
either $\varphi=0$, or $-\partial_{\gamma_{g}^{1}}b\geq0,\ -\partial
_{\gamma_{g}^{2}}\left(  b_{2}+\tau\right)  $ $\geq0$ there. The boundary
integral portion in the interior of $\Omega$ is also non-negative, because
there we have
\begin{gather*}
b=b_{2}+\tau\ \ \ \text{(and }b\geq b_{2}+\tau\ \ \text{in }\Omega_{1}\left(
\tau\right)  \text{)}\\
-\partial_{\gamma_{g}^{1}}b\ -\partial_{\gamma_{g}^{2}}\left(  b_{2}%
+\tau\right)  =\partial_{\gamma_{g}^{2}}b\ -\partial_{\gamma_{g}^{2}}\left(
b_{2}+\tau\right)  \geq0.
\end{gather*}
Taking the limit along the (Sard) sequence of $\tau$ going to $0,$ we obtain
$\Omega_{1}\left(  \tau\right)  \rightarrow\Omega$ up to a set of measure
zero, and
\begin{align*}
&  \int_{B_{4}}-\left\langle \nabla_{g}\varphi,\nabla_{g}b\right\rangle
_{g}dv_{g}\\
&  =\int_{\Omega}-\left\langle \nabla_{g}\varphi,\nabla_{g}b\right\rangle
_{g}dv_{g}\geq\frac{1}{3}\int_{\Omega}\left\vert \nabla_{g}b\right\vert
^{2}dv_{g}\\
&  =\frac{1}{3}\int_{B_{4}}\left\vert \nabla_{g}b\right\vert ^{2}dv_{g}.
\end{align*}

Case 2. $S$ has positive measure. The discriminant
\[
\mathcal{D}=\left(  \lambda_{1}-\lambda_{2}\right)  ^{2}\left(  \lambda
_{2}-\lambda_{3}\right)  ^{2}\left(  \lambda_{3}-\lambda_{1}\right)  ^{2}%
\]
is an analytic function in $B_{4},$ because the smooth $u$ is actually
analytic (cf. [M, p. 203]). So $\mathcal{D}$ must vanish identically. Then we
have either $\lambda_{1}\left(  x\right)  =\lambda_{2}\left(  x\right)  $ or
$\lambda_{2}\left(  x\right)  =\lambda_{3}\left(  x\right)  $ at any point
$x\in B_{4}.$ In turn, we know that $\lambda_{1}\left(  x\right)  =\lambda
_{2}\left(  x\right)  =\lambda_{3}\left(  x\right)  =\tan\left(  \frac{\pi}%
{6}\right)  $ and $b=K>b_{1}\left(  x\right)  $ at every \textquotedblleft
boundary\textquotedblright\ point of $S$ inside $B_{4},$ $x\in\partial
S\cap\mathring{B}_{4}.$ If the \textquotedblleft boundary\textquotedblright%
\ set $\partial S$ has positive measure, then $\lambda_{1}\left(  x\right)
=\lambda_{2}\left(  x\right)  =\lambda_{3}\left(  x\right)  =\tan\left(
\frac{\pi}{6}\right)  $ everywhere by the analyticity of $u,$ and
(\ref{IJacobi-3}) is trivially true. In the case that $\partial S$ has zero
measure, $b=b_{1}>K$ is smooth up to the boundary of every component of
$\left\{  x|\ b\left(  x\right)  >K\right\}  .$ By the pointwise Jacobi
inequality (\ref{jacobi-b2}), the integral inequality (\ref{IJacobi-3}) is
also valid in case 2.
\end{proof}

\ 

\section{Proof Of Theorem 1.1}

We assume that $R=4$ and $u$ is a solution on $B_{4}\subset\mathbb{R}^{3}$ for
simplicity of notation. By scaling $v\left(  x\right)  =u\left(  \frac{R}%
{4}x\right)  /\left(  \frac{R}{4}\right)  ^{2},$ we still get the estimate in
Theorem 1.1. Without loss of generality, we assume that the continuous Hessian
$D^{2}u$ sits on the convex branch of $\left\{  \left(  \lambda_{1}%
,\lambda_{2},\lambda_{3}\right)  |\ \lambda_{1}\lambda_{2}+\lambda_{2}%
\lambda_{3}+\lambda_{3}\lambda_{1}=1\right\}  $ containing $\left(
1,1,1\right)  /\sqrt{3},$ then $u$ satisfies (\ref{EsLag}) with $n=3$ and
$\Theta=\pi/2.$ By symmetry this also covers the concave branch corresponding
to $\Theta=-\pi/2.$

Step 1. By the integral Jacobi inequality (\ref{IJacobi-3}) in Proposition
\ref{PIJacobi}, $b$\ is subharmonic in the integral sense, then $b^{3}$ is
also subharmonic in the integral sense on the minimal surface $\mathfrak{M}%
=\left(  x,Du\right)  :$%
\begin{align*}
\int-\left\langle \nabla_{g}\varphi,\nabla_{g}b^{3}\right\rangle _{g}dv_{g} &
=\int-\left\langle \nabla_{g}\left(  3b^{2}\varphi\right)  -6b\varphi
\nabla_{g}b,\nabla_{g}b\right\rangle _{g}dv_{g}\\
&  \geq\int\left(  \varphi b^{2}\left\vert \nabla_{g}b\right\vert
^{2}+6b\varphi\left\vert \nabla_{g}b\right\vert ^{2}\right)  dv_{g}\geq0
\end{align*}
for all non-negative $\varphi\in C_{0}^{\infty},$ approximating $b$ by smooth
functions if necessary.

Applying Michael-Simon's mean value inequality [MS, Theorem 3.4] to the
Lipschitz subharmonic function $b^{3},$ we obtain
\[
b\left(  0\right)  \leq C\left(  3\right)  \left(  \int_{\mathfrak{B}_{1}%
\cap\mathfrak{M}}b^{3}dv_{g}\right)  ^{1/3}\leq C\left(  3\right)  \left(
\int_{B_{1}}b^{3}dv_{g}\right)  ^{1/3},
\]
where $\mathfrak{B}_{r}$ is the ball with radius $r$ and center $\left(
0,Du\left(  0\right)  \right)  $ in $\mathbb{R}^{3}\times\mathbb{R}^{3}$, and
$B_{r}$ is the ball with radius $r$ and center $0$ in $\mathbb{R}^{3}.$ Choose
a cut-off function $\varphi\in C_{0}^{\infty}\left(  B_{2}\right)  $ such that
$\varphi\geq0,$ $\varphi=1$ on $B_{1},$ and $\left\vert D\varphi\right\vert
\leq1.1,$ we then have
\[
\left(  \int_{B_{1}}b^{3}dv_{g}\right)  ^{1/3}\leq\left(  \int_{B_{2}}%
\varphi^{6}b^{3}dv_{g}\right)  ^{1/3}=\left(  \int_{B_{2}}\left(  \varphi
b^{1/2}\right)  ^{6}dv_{g}\right)  ^{1/3}.
\]
Applying the Sobolev inequality on the minimal surface $\mathfrak{M}$ [MS,
Theorem 2.1] or [A, Theorem 7.3] to $\varphi b^{1/2},$ which we may assume to
be $C^{1}$ by approximation, we obtain
\[
\left(  \int_{B_{2}}\left(  \varphi b^{1/2}\right)  ^{6}dv_{g}\right)
^{1/3}\leq C\left(  3\right)  \int_{B_{2}}\left\vert \nabla_{g}\left(  \varphi
b^{1/2}\right)  \right\vert ^{2}dv_{g}.
\]
Splitting the integrand as follows
\begin{align*}
\left\vert \nabla_{g}\left(  \varphi b^{1/2}\right)  \right\vert ^{2}  &
=\left\vert \frac{1}{2b^{1/2}}\varphi\nabla_{g}b+b^{1/2}\nabla_{g}%
\varphi\right\vert ^{2}\leq\frac{1}{2b}\varphi^{2}\left\vert \nabla
_{g}b\right\vert ^{2}+2b\left\vert \nabla_{g}\varphi\right\vert ^{2}\\
&  \leq\frac{1}{2}\varphi^{2}\left\vert \nabla_{g}b\right\vert ^{2}%
+2b\left\vert \nabla_{g}\varphi\right\vert ^{2},
\end{align*}
where we used $b\geq1,$ we get
\begin{align*}
b\left(  0\right)   &  \leq C\left(  3\right)  \int_{B_{2}}\left\vert
\nabla_{g}\left(  \varphi b^{1/2}\right)  \right\vert ^{2}dv_{g}\\
&  \leq C\left(  3\right)  \left(  \int_{B_{2}}\varphi^{2}\left\vert
\nabla_{g}b\right\vert ^{2}dv_{g}+\int_{B_{2}}b\left\vert \nabla_{g}%
\varphi\right\vert ^{2}dv_{g}\right) \\
&  \leq\underset{\text{Step\ 2}}{\underbrace{C\left(  3\right)  \left\Vert
Du\right\Vert _{L^{\infty}\left(  B_{2}\right)  }}}+C\left(  3\right)
\underset{\text{step\ 3}}{\underbrace{\left[  \left\Vert Du\right\Vert
_{L^{\infty}\left(  B_{3}\right)  }^{2}+\left\Vert Du\right\Vert _{L^{\infty
}\left(  B_{4}\right)  }^{3}\right]  }}.
\end{align*}

Step 2. By (\ref{IJacobi-3}) in Proposition \ref{PIJacobi}, $b$ satisfies the
Jacobi inequality in the integral sense:
\[
3\bigtriangleup_{g}b\geq\left\vert \nabla_{g}b\right\vert ^{2}.
\]
Multiplying both sides by the above non-negative cut-off function $\varphi\in
C_{0}^{\infty}\left(  B_{2}\right)  ,$ then integrating, we obtain
\begin{align*}
\int_{B_{2}}\varphi^{2}\left\vert \nabla_{g}b\right\vert ^{2}dv_{g} &
\leq3\int_{B_{2}}\varphi^{2}\bigtriangleup_{g}bdv_{g}\\
&  =-3\int_{B_{2}}\left\langle 2\varphi\nabla_{g}\varphi,\nabla_{g}%
b\right\rangle dv_{g}\\
&  \leq\frac{1}{2}\int_{B_{2}}\varphi^{2}\left\vert \nabla_{g}b\right\vert
^{2}dv_{g}+18\int_{B_{2}}\left\vert \nabla_{g}\varphi\right\vert ^{2}dv_{g}.
\end{align*}
It follows that
\[
\int_{B_{2}}\varphi^{2}\left\vert \nabla_{g}b\right\vert ^{2}dv_{g}\leq
36\int_{B_{2}}\left\vert \nabla_{g}\varphi\right\vert ^{2}dv_{g}.
\]
Observe the (\textquotedblleft conformality\textquotedblright)\ identity:
\[
\left(  \frac{1}{1+\lambda_{1}^{2}},\frac{1}{1+\lambda_{2}^{2}},\frac
{1}{1+\lambda_{3}^{2}}\right)  V=\left(  \sigma_{1}-\lambda_{1,}\text{
\ }\sigma_{1}-\lambda_{2,}\text{ \ }\sigma_{1}-\lambda_{3}\right)
\]
where we used the identity $V=%
%TCIMACRO{\dprod \limits_{i=1}^{3}}%
%BeginExpansion
{\displaystyle\prod\limits_{i=1}^{3}}
%EndExpansion
\sqrt{\left(  1+\lambda_{i}^{2}\right)  }=\sigma_{1}-\sigma_{3}$ with
$\sigma_{2}=1$. \ We then have
\begin{align}
\left\vert \nabla_{g}\varphi\right\vert ^{2}dv_{g} &  =\sum_{i=1}^{3}%
\frac{\left(  D_{i}\varphi\right)  ^{2}}{1+\lambda_{i}^{2}}Vdx=\sum_{i=1}%
^{3}\left(  D_{i}\varphi\right)  ^{2}\left(  \sigma_{1}-\lambda_{i}\right)
dx\label{conformality}\\
&  \leq2.42\bigtriangleup u\ dx.\nonumber
\end{align}
Thus
\begin{gather*}
\int_{B_{2}}\varphi^{2}\left\vert \nabla_{g}b\right\vert ^{2}dv_{g}\leq
C\left(  3\right)  \int_{B_{2}}\bigtriangleup u\ dx\\
\leq C\left(  3\right)  \left\Vert Du\right\Vert _{L^{\infty}\left(
B_{2}\right)  }.
\end{gather*}

Step 3. By (\ref{conformality}), we get
\[
\int_{B_{2}}b\left\vert \nabla_{g}\varphi\right\vert ^{2}dv_{g}\leq C\left(
3\right)  \int_{B_{2}}b\bigtriangleup u\ dx.
\]
Choose another cut-off function $\psi\in C_{0}^{\infty}\left(  B_{3}\right)  $
such that $\psi\geq0,$ $\psi=1$ on $B_{2},$ and $\left\vert D\psi\right\vert
\leq1.1.$ We have
\begin{align*}
\int_{B_{2}}b\bigtriangleup udx &  \leq\int_{B_{3}}\psi b\bigtriangleup
udx=\int_{B_{3}}-\left\langle bD\psi+\psi Db,Du\right\rangle dx\\
&  \leq\left\Vert Du\right\Vert _{L^{\infty}\left(  B_{3}\right)  }\int
_{B_{3}}\left(  b\left\vert D\psi\right\vert +\psi\left\vert Db\right\vert
\right)  dx\\
&  \leq C\left(  3\right)  \left\Vert Du\right\Vert _{L^{\infty}\left(
B_{3}\right)  }\int_{B_{3}}\left(  b+\left\vert Db\right\vert \right)  dx.
\end{align*}
Now
\[
b=\max\left\{  \ln\sqrt{1+\lambda_{\max}^{2}},\ K\right\}  \leq\lambda_{\max
}+K<\lambda_{1}+\lambda_{2}+\lambda_{3}+K=\bigtriangleup u+K,
\]
where $\lambda_{2}+\lambda_{3}>0$ follows from $\arctan\lambda_{2}%
+\arctan\lambda_{3}=\frac{\pi}{2}-\arctan\lambda_{1}>0.$ Hence
\[
\int_{B_{3}}bdx\leq C(3)(1+\left\Vert Du\right\Vert _{L^{\infty}\left(
B_{3}\right)  }).
\]
And we have left to estimate $\int_{B_{3}}\left\vert Db\right\vert dx:$
\begin{align*}
\ \int_{B_{3}}\left\vert Db\right\vert dx &  \leq\int_{B_{3}}\sqrt{\sum
_{i=1}^{3}\frac{\left(  b_{i}\right)  ^{2}}{\left(  1+\lambda_{i}^{2}\right)
}\left(  1+\lambda_{1}^{2}\right)  \left(  1+\lambda_{2}^{2}\right)  \left(
1+\lambda_{3}^{2}\right)  }\ dx\\
&  =\int_{B_{3}}\left\vert \nabla_{g}b\right\vert Vdx\\
&  \leq\left(  \int_{B_{3}}\left\vert \nabla_{g}b\right\vert ^{2}Vdx\right)
^{1/2}\left(  \int_{B_{3}}Vdx\right)  ^{1/2}.
\end{align*}
Repeating the \textquotedblleft Jacobi\textquotedblright\ argument from Step
2, we see
\[
\int_{B_{3}}\left\vert \nabla_{g}b\right\vert ^{2}Vdx\leq C\left(  3\right)
\left\Vert Du\right\Vert _{L^{\infty}\left(  B_{4}\right)  }.
\]
Then by the Sobolev inequality on the minimal surface $\mathfrak{M},$ we have
\[
\int_{B_{3}}Vdx=\int_{B_{3}}dv_{g}\leq\int_{B_{4}}\phi^{6}dv_{g}\leq C\left(
3\right)  \left(  \int_{B_{4}}\left\vert \nabla_{g}\phi\right\vert ^{2}%
dv_{g}\right)  ^{3},
\]
where the non-negative cut-off function $\phi\in C_{0}^{\infty}\left(
B_{4}\right)  $ satisfies $\phi=1$ on $B_{3},$ and $\left\vert D\phi
\right\vert \leq1.1.$ Applying the conformality equality (\ref{conformality})
again, we obtain
\[
\int_{B_{4}}\left\vert \nabla_{g}\phi\right\vert ^{2}dv_{g}\leq C\left(
3\right)  \int_{B_{4}}\bigtriangleup u\ dx\leq C\left(  3\right)  \left\Vert
Du\right\Vert _{L^{\infty}\left(  B_{4}\right)  }.
\]
Thus we get
\[
\int_{B_{3}}Vdx\leq C\left(  3\right)  \left\Vert Du\right\Vert _{L^{\infty
}\left(  B_{4}\right)  }^{3}%
\]
and
\[
\int_{B_{3}}\left\vert Db\right\vert dx\leq C\left(  3\right)  \left\Vert
Du\right\Vert _{L^{\infty}\left(  B_{4}\right)  }^{2}.
\]
In turn, we obtain
\[
\int_{B_{2}}b\left\vert \nabla_{g}\varphi\right\vert ^{2}dv_{g}\leq C\left(
3\right)  \left[  K\left\Vert Du\right\Vert _{L^{\infty}\left(  B_{3}\right)
}+\left\Vert Du\right\Vert _{L^{\infty}\left(  B_{3}\right)  }^{2}+\left\Vert
Du\right\Vert _{L^{\infty}\left(  B_{4}\right)  }^{3}\right]  .
\]
Finally collecting all the estimates in the above three steps, we arrive at
\begin{align*}
\lambda_{\max}\left(  0\right)   &  \leq\exp\left[  C\left(  3\right)  \left(
\left\Vert Du\right\Vert _{L^{\infty}\left(  B_{4}\right)  }+\left\Vert
Du\right\Vert _{L^{\infty}\left(  B_{4}\right)  }^{2}+\left\Vert Du\right\Vert
_{L^{\infty}\left(  B_{4}\right)  }^{3}\right)  \right]  \\
&  \leq C\left(  3\right)  \exp\left[  C\left(  3\right)  \left\Vert
Du\right\Vert _{L^{\infty}\left(  B_{4}\right)  }^{3}\right]  .
\end{align*}
This completes the proof of Theorem 1.1.

\textbf{Remark. }A sharper Hessian estimate and a gradient estimate for the
special Lagrangian equation (\ref{EsLag}) with $n=2$ were derived by
elementary method in [WY1].  More involved arguments are needed to obtain the
Hessian and gradient estimates for  (\ref{EsLag}) with $n=3$ and $\left\vert
\Theta\right\vert >\pi/2$ in [WY2].


\bigskip

\begin{thebibliography}{999}                                                                                              %


\bibitem[A]{A}Allard, William K., \emph{On the first variation of a varifold.}
Ann. of Math. (2) \textbf{95} (1972), 417--491.

\bibitem[BC]{BC}Bao, Jiguang and Chen, Jingyi, \emph{Optimal regularity for
convex strong solutions of special Lagrangian equations in dimension 3.}
Indiana Univ. Math. J. \textbf{52} (2003), 1231--1249.

\bibitem[CW]{CW}Chou, Kai-Seng and Wang, Xu-Jia, \emph{A variational theory of
the Hessian equation.} Comm. Pure Appl. Math. \textbf{54} (2001), 1029--1064.

\bibitem[HL]{HL}Harvey, Reese and Lawson, H. Blaine. Jr., \emph{Calibrated
geometry.} Acta Math. \textbf{148} (1982), 47--157.

\bibitem[H]{H}Heinz, Erhard, \emph{On elliptic Monge-Amp\`{e}re equations and
Weyl's embedding problem}. J. Analyse Math. \textbf{7} 1959 1--52.

\bibitem[M]{M}Morrey, Charles B., Jr., \emph{On the analyticity of the
solutions of analytic non-linear elliptic systems of partial differential
equations. I. Analyticity in the interior.} Amer. J. Math. \textbf{80} (1958) 198--218.

\bibitem[MS]{MS}Michael, James H. and Simon, Leon M., \emph{Sobolev and
mean-value inequalities on generalized submanifolds of }$\mathbb{R}^{n}%
$\emph{.} Comm. Pure Appl. Math. \textbf{26} (1973), 361--379.

\bibitem[P]{P}Pogorelov, Aleksei Vasil'evich, \emph{The Minkowski
multidimensional problem.} Translated from the Russian by Vladimir Oliker.
Introduction by Louis Nirenberg. Scripta Series in Mathematics. V. H. Winston
\& Sons, Washington, D.C.; Halsted Press [John Wiley \& Sons], New
York-Toronto-London, 1978.

\bibitem[T]{T}Trudinger, Neil S., \emph{Weak solutions of Hessian equations.}
Comm. Partial Differential Equations \textbf{22} (1997), no. 7-8, 1251--1261.

\bibitem[U1]{U1}Urbas, John I. E., \emph{On the existence of nonclassical
solutions for two classes of fully nonlinear elliptic equations.} Indiana
Univ. Math. J. \textbf{39} (1990), no. 2, 355--382.

\bibitem[U2]{U2}Urbas, John, \emph{Some interior regularity results for
solutions of Hessian equations.} Calc. Var. Partial Differential Equations
\textbf{11} (2000), 1--31.

\bibitem[U3]{U3}Urbas, John, \emph{An interior second derivative bound for
solutions of Hessian equations.} Calc. Var. Partial Differential Equations
\textbf{12} (2001), 417--431.

\bibitem[WY1]{WY1}Warren, Micah and Yuan, Yu, \emph{Explicit gradient
estimates for minimal Lagrangian surfaces of dimension two.} preprint.

\bibitem[WY2]{WY2}Warren, Micah and Yuan, Yu, \emph{Hessian and gradient
estimates for three dimensional special\ Lagrangian equations with large
phase}. preprint.

\bibitem[Y1]{Y1}Yuan, Yu, \emph{A Bernstein problem for special Lagrangian
equations.} Invent. Math. \textbf{150} (2002), 117--125.

\bibitem[Y2]{Y2}Yuan, Yu, \emph{Global solutions to special Lagrangian
equations.} Proc. Amer. Math. Soc. \textbf{134} (2006), no. 5, 1355--1358.
\end{thebibliography}
\end{document}